\documentclass[11pt]{article}

\usepackage{amsmath,amssymb,amsthm,enumerate}

\newtheorem{lem}{Lemma}
\newtheorem{theorem}{Theorem}

\newcommand{\prob}{{\mathbb P}}
\newcommand{\E}{{\mathbb E}}
\newcommand{\reals}{{\mathbb R}}
\newcommand{\nats}{{\mathbb N}}
\newcommand{\probto}{{\overset{p}{\longrightarrow}}}

\begin{document}

\title{The Proportion of the Population Never Hearing a Rumour}
\author{Y. Duan and A. Ganesh}

\maketitle

\begin{abstract}
Sudbury~\cite{sudbury} showed for the Maki-Thompson model of rumour spreading that the proportion of the population 
never hearing the rumour converges in probability to a limiting constant (approximately equal to $0.203$) as the population 
size tends to infinity. 
We extend the analysis to a generalisation of the Maki-Thompson model. 

\emph{Keywords} Rumour spreading, epidemic processes.

{\bf Subject classification} 60G42, 60G50.
\end{abstract}

\section{Introduction}
The following model of rumour spreading was introducing by Maki and Thompson~\cite{maki73}, as a variant of an earlier 
model of Daley and Kendall~\cite{daley65}: there is a population of size $n$, some of whom initially know a rumour, and 
are referred to as infected. Time is discrete. In each time step, an infected individual chosen uniformly at random (or arbitrarily) 
contacts a member of the population chosen uniformly at random (including itself). If this individual hasn't yet heard the rumour 
(is susceptible), then the contacted individual becomes infected; otherwise, the contacting individual loses interest in spreading 
the rumour and is termed removed (but remains in the population and can be contacted by other infectives. In the Daley-Kendall 
model, if an infective contacts another infective, both become removed, whereas, in the Maki-Thompson model, only the 
initiator of the contact is removed.) The process ends when there are no more infectives. 
A natural question to ask is how many individuals remain susceptible at this terminal time, and consequently never hear 
the rumour. It was shown by Sudbury~\cite{sudbury} that in the large population limit of $n$ tending to infinity, the random 
proportion of the population never hearing the rumour converges in probability to a limiting constant. 

We consider the following generalisation of the Maki-Thompson model: each infective loses interest in spreading the rumour 
(and becomes removed) after $k$ failed attempts, i.e., after contacting infected or removed individuals $k$ times. Here, 
$k\geq 1$ is a specified constant, which is a parameter of the model; if $k=1$, we recover the original model. Our main result 
is as follows.

\begin{theorem} \label{thm:probconv}
Consider the generalisation of the Maki-Thompson model described above, parametrised by $k$ and starting with a single 
infective and $n-1$ susceptibles. Let $S_{\infty}$ denote the number of susceptibles when the process terminates, i.e., 
when the number of infectives hits zero. Then,
$$
\frac{S_{\infty}}{n} \, \probto \, y^* \mbox{ as $n\to \infty$}, 
$$
where $y^*$ is the unique solution in $(0,1)$ of the equation $(k+1)(1-y)=-\log y$, and logarithms are natural unless 
specified otherwise.
\end{theorem}

The proof is presented in the next section. We observe that $y^*=y^*(k)$ is a decreasing function of $k$, and is 
well-approximated by $e^{-(k+1)}$ for large $k$. This tells us that, qualitatively, the proportion of the population not 
hearing a rumour decays exponentially in the number of failed attempts before agents lose interest in spreading the rumour.

Pittel~\cite{pittel} showed in the Maki-Thompson model that the proportion of nodes not hearing the rumour, suitably centred 
and rescaled, converges in distribution to a normal random variable. An extension of this result to our generalised model is an 
open problem.

\section{Model and Analysis}
Denote by $S_t$ the number of susceptibles present in time slot $t$. If at least one infective is present during this time slot, 
then there is an infection attempt during this time slot, which succeeds with probability $S_t/n$ (or $S(t)/(n-1)$ if an infective 
never contacts itself; the distinction is immaterial for large $n$). In that case, $S(t+1)=S(t)-1$. Otherwise, the number of 
failure attempts associated with the infective node which initiated the contact is incremented by 1; if its value becomes equal 
to $k$, the infective node becomes removed. We could describe this process as a Markov chain by keeping track of 
$I^0_t, I^1_t, \ldots, I^{k-1}_t$, which denote respectively the number of infective nodes which have seen $0,1,\ldots,k-1$ 
failed infection attempts. A simpler Markovian representation is obtained by keeping track of $I_t$, the number of infection 
attempts avaible in time step $t$, which increases by $k$ whenever a new node is infected. We initialise the process with 
$S_0=n-1$ and $I_0=k$; the process terminates when $I_t$ hits zero for the first time. If $I_t>0$, then 
\begin{equation} \label{eq:trans_prob}
(S_{t+1},I_{t+1}) = \begin{cases}
(S_t-1,I_t+k), & \mbox{w.p. } \; S_t/n, \\
(S_t, I_t-1), & \mbox{w.p. } \; 1-(S_t/n), \\
\end{cases}
\end{equation}
where we use the abbreviation w.p. for ``with probability". 

Let $T$ denote the random time that the process terminates, i.e, when $I_t$ hits zero for the first time. We see from 
(\ref{eq:trans_prob}) that $(k+1)S_t+I_t+t$ is conserved. Hence,
$$
(k+1)S_T + T = (k+1)S_0+I_0+0 = (k+1)(n-1) + k,
$$
so that
$$
T= \inf \{ t: (k+1)(n-1-S_t) \leq t-k \} = \inf \{ t: (k+1)(n-S_t) \leq t+1 \}.
$$
Define $\tilde S_t$, $t=0,1,2,\ldots$ to be a Markov process on the state space $\{ 0,1,\ldots,n-1 \}$ with transition 
probabilities
\begin{equation} \label{eq:trans_probs}
p_{s,s} = 1-\frac{s}{n}, \quad p_{s,s-1} = \frac{s}{n}, \quad 0\leq s\leq n,
\end{equation}
with initial condition $\tilde S_0=n-1$. Then $\tilde S_t$ and $S_t$ have the same transition probabilities while $I_t$ is non-zero; 
hence, it is clear that we can couple the processes $S_t$ and $\tilde S_t$ in such a way that they are equal until the random 
time $T$. Consequently, we can write 
\begin{equation} \label{eq:terminal_time}
T= \inf \{ t: (k+1)(n-\tilde S_t) \leq t+1 \},
\end{equation}
which relates $T$ to a level crossing time of a lazy random walk. As the random walk $\tilde S_t$ is non-increasing, $S_T$ is 
explicitly determined by $T$; we have
\begin{equation} \label{eq:terminal_number}
S_T = \tilde S_T = n- \frac{T+1}{k+1}.
\end{equation}
While it is possible to study the random variable $T$ directly by analysing the random walk $\tilde S_t$, we will follow the work 
of Sudbury~\cite{sudbury} and consider a somewhat indirect approach. The random walk $\tilde S_t$ is exactly the same as 
the random walk $s_k$ in that paper, but the level-crossing required for stopping is different.

Define the filtration ${\cal F}_t= \sigma(\tilde S_u, 1\leq u\leq t)$, $t\in \nats$, and notice that the random time $T$ defined 
in (\ref{eq:terminal_time}) is a stopping time, i.e., the event $\{ T\leq t \}$ is ${\cal F}_t$-measurable. Moreover, $T$ is 
bounded by $(k+1)n$. Let
$$
M_1(t) = \left( \frac{n}{n-1} \right )^t \tilde S_t, \quad M_2(t) = \left( \frac{n}{n-2} \right )^t \tilde S_t(\tilde S_t-1).
$$
The lemma below is an exact analogue of a corresponding result in~\cite{sudbury} and follows easily from the transition 
probabilities in (\ref{eq:trans_probs}), so the proof is omitted.

\begin{lem} \label{lem_martingales}
The processes $M_1(t\wedge T)$ and $M_2(t\wedge T)$ are ${\cal F}_t$-martingales. 
\end{lem}

Applying the optional stopping theorem (OST) to $M_1(t\wedge T)$, we get 
\begin{equation} \label{eq:conv_mart1}
\E \Bigl[ \Bigl( \frac{n}{n-1} \Bigr)^T \tilde S_T \Bigr] = \tilde S_0.
\end{equation}
We show that for large $n$ the above random variables concentrate around their mean values and, after suitable rescaling, 
converging in probability.

\begin{lem} \label{lem:suscep_plim}
Let $\tilde S_T$ denotes the final number of susceptibles and $T$ the random time (number of attempts to spread the rumour) 
after which the process terminates in a population of size $n$. The dependence of $T$ and $\tilde S_T$ on $n$ has been 
suppressed in the notation. Then, 
$$
\Bigl( \frac{n}{n-1} \Bigr)^T \frac{\tilde S_T}{n} \probto 1 \mbox{ as $n\to \infty$.}
$$ 
\end{lem}

\begin{proof}
The proof is largely reproduced from~\cite{sudbury} but is included for completeness. It proceeds by bounding the variance 
of the random variables of interest and invoking Chebyshev's inequality. 
We have by (\ref{eq:conv_mart1}) that
\begin{equation*} 
\mathrm{Var} \Bigl\{ \Bigl( \frac{n}{n-1} \Bigr)^{T} \tilde S_T \Bigr\}
= \E \Bigl[ \Bigl( \frac{n-1}{n} \Bigr)^{-2T} \tilde S_T^2 \Bigr] - \tilde S_0^2, 
\end{equation*} 
whereas, applying the OST to $M_2(t\wedge T)$, we get 
\begin{equation*} 
\E \Bigl[ \Bigl( \frac{n}{n-2} \Bigr)^T (\tilde S_T^2-\tilde S_T) \Bigr] = \tilde S_0^2-\tilde S_0.
\end{equation*}
Combining the last two equations, we can write 
\begin{equation*} 
\begin{aligned}
& \mathrm{Var} \Bigl\{ \Bigl( \frac{n}{n-1} \Bigr) ^{T} \tilde S_T \Bigr\} \\
&= \E \Bigl[ \Bigl( \frac{n-1}{n} \Bigr)^{-2T} \tilde S_T^2 \Bigr] - 
\E \Bigl[ \Bigl( \frac{n}{n-2} \Bigr)^T (\tilde S_T^2 - \tilde S_T) \Bigr] -\tilde S_0 \\
&= \E \Bigl\{ \Bigl[ \Bigl( \frac{n-1}{n} \Bigr)^{-2T} - \Bigl( \frac{n-2}{n} \Bigr)^{-T} \Bigr] \tilde S_T^2 \Bigr\} 
+ \E \Bigl[ \Bigl( \frac{n}{n-2} \Bigr)^T \tilde S_T \Bigr] - \tilde S_0.
\end{aligned}
\end{equation*} 
Now, the first term in the above sum is negative, since $(1-\frac{1}{n})^2>1-\frac{2}{n}$. Next, since $T$ is bounded 
above by $(k+1)n$, we have
\begin{equation*} 
\begin{aligned} 
\mathrm{Var} \Bigl\{ \Bigl( \frac{n}{n-1} \Bigr)^{T} \tilde S_T \Bigr\} 
&< \E \Bigl[ \Bigl( \frac{n}{n-1} \frac{n-1}{n-2} \Bigr)^{T} \tilde S_T \Bigr] - \tilde S_0 \\
&\leq \Bigl( \frac{n-1}{n-2} \Bigr)^{(k+1)n} \E \Bigl[ \Bigl( \frac{n}{n-1} \Bigr)^{T} \tilde S_T \Bigr]-\tilde S_0 \\
&\sim (e^{k+1}-1) \tilde S_0,
\end{aligned}
\end{equation*}
where we have used the fact that $\E \bigl[ (\frac{n}{n-1})^T \tilde S_T \bigr] =\tilde S_0$ to obtain the asymptotic 
equivalence on the last line. (Recall that, for sequences $x_n$ and $y_n$, we write $x_n\sim y_n$ to denote that 
$x_n/y_n \to 1$ as $n\to \infty$.) Thus, we conclude that
$$
\mathrm{Var} \Bigl\{ \Bigl( \frac{n}{n-1} \Bigr)^{T} \frac{\tilde S_T}{n} \Bigr\} \leq 
\frac{ (e^{k+1}-1) \tilde S_0}{n^2},
$$
which tends to zero as $n$ tends to infinity, since $\tilde S_0=n-1$. The claim of the lemma now follows from 
(\ref{eq:conv_mart1}) and Chebyshev's inequality.
\end{proof}

Consider the sequence of random vectors $\bigl( \frac{T}{n}, \frac{\tilde S_T}{n} \bigr)$, which take values in the compact set 
$K=[0,k+1]\times [0,1]$; the dependence of $T$ and $\tilde S_T$ on $n$ has not been made explicit in the notation. 
Define $f:K\to \reals^2$ by
\begin{equation} \label{eq:function}
f(x,y) = \Bigl( \frac{x}{k+1}+y-1, e^x y-1 \Bigr).
\end{equation}
Then we see from (\ref{eq:terminal_number}) and Lemma~\ref{lem:suscep_plim} that
\begin{equation} \label{eq:plim_function}
f ( T/n, \tilde S_T/n ) \probto (0,0) \mbox{ as $n\to \infty$.}
\end{equation}
We want to use this to prove convergence in probability of the sequences $T/n$ and $\tilde S_T/n$.

Firstly, we observe that if $f(x,y)=(0,0)$, then $y$ solves the equation $(k+1)(1-y) + \log y=0$, and $x=(k+1)(1-y)$. 
The function $y\mapsto (k+1)(1-y)+\log y$ is strictly concave and is zero at $y=1$; by considering its derivative at 1 
and its value near 0, it can be seen that the function has one other zero, which lies in $(0,1)$. Call this value $y^*$ and 
define $x^*=(k+1)(1-y^*)$. We now have the following.

\begin{lem} \label{lem:convprob_image}
Fix $\delta>0$. Then, as $n$ tends to infinity,
$$
\prob \bigl( ( T/n, \tilde S_T/n ) \notin B_{\delta}(0,1) \cup B_{\delta}(x^*,y^*) \bigr) \to 0,
$$
where $B_{\delta}(x,y)$ denotes the open ball of radius $\delta$ centred on $(x,y)$.
\end{lem}

\begin{proof}
Suppose this is not the case. Then, there is an $\alpha>0$ and infinitely many $n$ such that 
$$
\prob \bigl( (T/n, \tilde S_T/n) \notin B_{\delta}(0,0) \cup B_{\delta}(x^*,y^*) \bigr) >\alpha.
$$
Since $f$ is continuous, so is its norm. Hence, its minimum on the compact set $K\setminus \{ B_{\delta}(0,0) 
\cup B_{\delta}(x^*,y^*) \}$ is attained, and must be strictly positive as $f$ has no zeros other than 
$(0,1)$ and $x^*,y^*)$. Hence, there is an $\epsilon>0$ such that $\| f(x,y)\|>\epsilon$ whenever 
$(x,y)\notin B_{\delta}(0,1) \cup B_{\delta}(x^*,y^*)$. Thus, we have shown that there are infinitely many 
$n$ such that 
$$
\prob ( \| f(T/n,\tilde S_T/n) \| > \epsilon) >\alpha,
$$
which contradicts (\ref{eq:plim_function}). This proves the claim of the lemma.
\end{proof}

Next, define $\tau_j =\inf \{ t: \tilde S_t = n-j \}$, $X_j= \tau_{j+1}-\tau_j$, and observe from (\ref{eq:trans_probs}) 
and the initial condition $\tilde S_0=n-1$ that 
\begin{equation} \label{eq:stilde_geom_sum}
\tau_1 = 0, \quad X_j \sim Geom \Bigl( \frac{n-j}{n} \Bigr),
\end{equation}
and that $X_j, j=1,\ldots,n-1$ are mutually independent; here, $\sim$ denotes equality in distribution. We also have from 
(\ref{eq:terminal_time}) that 
\begin{equation} \label{eq:terminal}
n-\tilde S_T = \inf \{ j: X_1+\ldots+X_j \geq (k+1)j \}.
\end{equation}
We now need the following elementary tail bound on the binomial distribution in order to complete the proof of 
Theorem~\ref{thm:probconv}.

\begin{lem} \label{lem:binom_ld}
Let $X$ be binomially distributed with parameters $n$ and $p$, denoted $X\sim Bin(n,p)$. Then, for any $q>p$, we have 
$$
\prob(X \geq nq) \leq \exp \Bigl( -n\Bigl[q\log\frac{q}{p}-q+p \Bigr] \Bigr).
$$
\end{lem}

\begin{proof} Recall the well-known large deviations bound,
$$
\prob(X \geq nq) \leq \exp (-nH(q;p)), \mbox{ where } H(q;p)=q\log\frac{q}{p}+(1-q)\log\frac{1-q}{1-p},
$$
which is a consequence of Sanov's theorem. This inequality, or slight variants, are known as Bernstein or Chernoff bounds. 

The claim of the lemma follows from the above inequality by noting that 
$$
(1-q)\log\frac{1-p}{1-q} \leq (1-q)\Bigl( \frac{1-p}{1-q}-1 \Bigr) = q-p,
$$
which follows from the inequality $\log x\leq x-1$.
\end{proof}

\begin{proof}[Proof of Theorem~\ref{thm:probconv}]
In vew of Lemma~\ref{lem:convprob_image}, it remains only to show, for some $\epsilon \in (0,y^*)$, that 
$\prob(\tilde S_T/n >1-\epsilon)$ tends to zero as $n$ tends to infinity. 

Fix $\epsilon>0$. For each $j\in \nats$,  let $Y^{(j)}_i$, $i\in \nats$ be iid random variables with a $Geom((n-j)/n)$ distribution. 
Then $Y^{(j)}_i$ stochastically dominates $X_i$ for every $i\leq j$, and we see from (\ref{eq:terminal}) that 
\begin{equation*}
\begin{aligned}
\prob(\tilde S_T/n \geq 1-\epsilon) &= \prob(\exists j\leq \epsilon n: X_1+\ldots+X_j \geq (k+1)j) \\
&\leq \sum_{j=1}^{\lfloor \epsilon n \rfloor} \prob(X_1+\ldots+X_j \geq (k+1)j) \\
&\leq \sum_{j=1}^{\lfloor \epsilon n \rfloor} \prob \bigl( Y^{(j)}_1+\ldots+Y^{(j)}_j \geq (k+1)j \bigr) \\
&\leq \sum_{j=1}^{\lfloor \epsilon n \rfloor} \prob \Bigl( Bin\Bigl( (k+1)j-1, \frac{n-j}{n} \Bigr) \leq j-1 \Bigr) \\
&= \sum_{j=1}^{\lfloor \epsilon n \rfloor} \prob \Bigl( Bin\Bigl( (k+1)j-1, \frac{j}{n} \Bigr) \geq kj \Bigr) \\
&\leq \sum_{j=1}^{\lfloor \epsilon n \rfloor} \prob \Bigl( Bin\Bigl( (k+1)j, \frac{j}{n} \Bigr) \geq kj \Bigr).
\end{aligned}
\end{equation*}
Hence, it follows from Lemma~\ref{lem:binom_ld} that, for $\epsilon<\frac{k}{k+1}$, we have 
\begin{equation*}
\begin{aligned}
\prob(\tilde S_T/n \geq 1-\epsilon) &\leq \sum_{j=1}^{\lfloor \epsilon n \rfloor} 
\exp \Bigl( -(k+1)j \Bigl[ \frac{k}{k+1}\log \frac{kn}{(k+1)j}-\frac{k}{k+1}+\frac{j}{n} \Bigr] \Bigr) \\
&\leq \sum_{j=1}^{\lfloor \epsilon n \rfloor} \exp \Bigl( -kj \log \frac{kn}{(k+1)j} +kj \Bigr) \\
&\leq \sum_{j=1}^{\lfloor \frac{k}{k+1}\frac{\sqrt{n}}{e}-1 \rfloor} n^{-kj/2} +
\sum_{j=\lfloor \frac{k}{k+1}\frac{\sqrt{n}}{e} \rfloor}^{\lfloor \epsilon n \rfloor} e^{-kj}.
\end{aligned}
\end{equation*}
It is easy to see that both sums above vanish as $n$ tends to infinity. This completes the proof of the theorem.
\end{proof}

\end{document}